\documentclass[a4paper,12pt]{article}
\usepackage{amsmath,amsfonts,amssymb,color}
\usepackage[english]{babel}
\usepackage{amsmath,amsfonts,amssymb,amscd,mathrsfs}
\usepackage{amsfonts}
\usepackage{amsthm}

\newtheorem{teo}{Theorem}[section]
\newtheorem{prop}[teo]{Proposition}
\newtheorem{lem}[teo]{Lemma}

\newtheorem{defi}[teo]{Definition}
\newtheorem{rem}[teo]{Remark}

\begin{document} 

\title{\vspace*{0cm} Riemannian metrics in infinite dimensional self-adjoint operator groups}

\date{}
\author{Alberto Manuel L\'opez Galv\'an\footnote{Supported by Instituto Argentino de Matem\'atica (CONICET).}}

\maketitle

\setlength{\parindent}{0cm} 

\begin{abstract}
The aim of this paper is the study of the geodesic distance in operator groups with several Riemannian metrics. More precisely we study the geodesic distance in self-adjoint operator groups with the left invariant Riemannian metric induced by the infinite trace and extend known results about the completeness of some classical Banach-Lie groups to this general class.
\end{abstract}

\section{Background and definitions}
In \cite{Lopez Galvan} we studied Riemannian metrics on an infinite dimensional symplectic group. There we proved that if we endow the tangent spaces with the classical left invariant metric then the metric space corresponding to the geodesic distance is complete. In this paper we will extend this result to a more general class of Riemannian operator groups, the self-adjoint operator groups. Examples of these self-adjoint operator groups are the classical unitary group ${\rm U}_2(\mathcal{H})$ and the symplectic group ${\rm Sp}_2(\mathcal{H})$. We will follow the notations and definitions of \cite{Larotonda} and \cite{Lopez Galvan}, so we recall some of the facts stated there.  

Let $\mathcal{B}(\mathcal{H})$ be the space of bounded operators acting on an infinite dimensional complex Hilbert space $\mathcal{H}$. Denote by $\mathcal{B}_2(\mathcal{H})$ the Hilbert-Schmidt class $$\mathcal{B}_2(\mathcal{H})=\left\{a \in \mathcal{B}(\mathcal{H}): Tr(a^*a)< \infty\right\}$$ where $Tr$ is the usual trace in $\mathcal{B}(\mathcal{H})$. This space is a Hilbert space with the inner product $$\langle a,b \rangle=Tr(b^*a).$$ The norm induced by this inner product is called the 2-norm and denoted by $$\Vert a \Vert_2=Tr(a^*a)^{1/2}.$$ The usual operator norm will be denoted by $\Vert \ \Vert$.
We use the subscript $h$ (resp. $ah$) to denote the sets of hermitic (resp. anti-hermitic) operators, e.g. $\mathcal{B}_2(\mathcal{H})_{ah}=\lbrace x\in \mathcal{B}_2(\mathcal{H}): x=-x^*\rbrace$.
In $\mathcal{B}(\mathcal{H})$ we have the classical general linear group  $${\rm GL_2}(\mathcal{H})=\lbrace g \in {\rm GL}(\mathcal{H}) : g-1 \in \mathcal{B}_2(\mathcal{H}) \rbrace.$$ This group has a differentiable structure when endowed with the metric $\Vert g_1 - g_2\Vert_2$ (note that $g_1 - g_2 \in \mathcal{B}_2(\mathcal{H})$); it is a Banach-Lie group with Banach-Lie algebra $\mathcal{B}_2(\mathcal{H})$.

Let ${\rm U}_2(\mathcal{H})\subset {\rm GL_2}(\mathcal{H})$ be the unitary group. It is known that ${\rm U}_2(\mathcal{H})$ is a Banach-Lie subgroup of $ {\rm GL_2}(\mathcal{H})$ with Banach-Lie algebra $\mathcal{B}_2({\mathcal{H}})_{ah}$.

\begin{defi}
Let $G$ be a connected abstract subgroup of ${\rm GL_2}(\mathcal{H})$. We say that $G$ is a self-adjoint subgroup if $g^* \in G$ whenever $g \in G$ (for short,  we write $G^* = G$). Note that a connected Banach-Lie group $G$ is
self-adjoint if and only if $\mathfrak{g^*} =\mathfrak{g}$, where $\mathfrak{g}$ denotes the Banach-Lie algebra of $G$.
\end{defi}

Throughout this paper, if $M$ is any submanifold of ${\rm GL_2}(\mathcal{H})$, we will denote by  $\mathfrak{b}(g,\cdot)$ the metric in each tangent space  $T_gM$. The length of a smooth curve measured with the metric $\mathfrak{b}$ will be denoted by $$L_{\mathfrak{b}}(\alpha)=\int_0^1 \mathfrak{b}(\alpha(t),\dot{\alpha}(t)) dt.$$ We define the geodesic distance between two points $p,q \in M\subseteq {\rm GL_2}(\mathcal{H})$ as the infimum of the length of all piecewise smooth curves in $M$ joining $p$ to $q$,  $$d_{\mathfrak{b}}(p,q)=\inf\left\{L_{\mathfrak{b}}(\alpha): \alpha\subset M, \alpha(0)=p, \ \alpha(1)=q \right\}.$$     
We will denote with $G$ a closed, connected self-adjoint Banach-Lie subgroup of $ {\rm GL_2}(\mathcal{H})$ and by $\mathfrak{g}\subset \mathcal{B}_2(\mathcal{H})$ its closed Banach-Lie algebra. Using the left action on itself, the tangent space at $g \in G$ is $$T_gG=g.\mathfrak{g}.$$ We introduce the classical left invariant metric for $v \in T_gG$  by 
\begin{equation}\label{metinvizq}
\mathcal{I}(g,v)=\Vert g^{-1}v\Vert_2. 
\end{equation}
One of the purposes of this paper is to prove that for such $G$ the metric space $(G,d_{\mathcal{I}})$ is complete. Using the polar decomposition we will define a mixed metric related to the unitary and positive part of the group. We also study the geodesic curves with these metrics and we will describe them in terms of exponentials of operators. 

If $A \subset {\rm GL}_2(\mathcal{H})$ is a set, $\langle A\rangle$ will denote the abstract subgroup generated by $A$ (the group whose elements are the inverses and the finite products of elements in $A$). 
Let ${\rm GL}_2^+(\mathcal{H}) := \lbrace g > 0 : g \in  {\rm GL_2}(\mathcal{H})\rbrace$ be the subset of positive invertible operators. It is known that ${\rm GL}_2^+(\mathcal{H})$ is a submanifold of the open set $\Delta=\lbrace \beta + X \in \mathbb{C}\oplus \mathcal{B}_2(\mathcal{H}) : \beta + X> 0\rbrace$. Let $\vert x \vert =(x^*x)^{1/2}=\exp(\frac{1}{2} \ln(x^*x))$ be the modulus operator, then it is clear that $\vert x \vert \in {\rm GL}_2^+(\mathcal{H}) \subset {\rm GL_2}(\mathcal{H})$ if $x\in 1+\mathcal{B}_2(\mathcal{H})$.

For $p \in {\rm GL}_2^+(\mathcal{H})$, we identify the tangent space $T_p {\rm GL}_2^+(\mathcal{H})$ with $\mathcal{B}_2(\mathcal{H})_h$ and endow this manifold with a complete Riemannian metric by means of the formula \begin{equation}\label{metpos}\mathfrak{p}(p,x)=\Vert p^{-1/2}xp^{-1/2}\Vert_2 \end{equation} for $p \in {\rm GL}_2^+(\mathcal{H})$ and $x \in T_p {\rm GL}_2^+(\mathcal{H})$. Euler's equation $\nabla_{\dot{\gamma}}\dot{\gamma}=0$ for the covariant derivative introduced by the Riemannian connection reads $\ddot{\gamma}=\dot{\gamma}\gamma^{-1}\dot{\gamma}$, and it is not hard to see that the unique solution of this equation with $\gamma(0)=p$ and $\gamma(1)=q$ is given by the smooth curve $$\gamma_{pq}(t)=p^{1/2}(p^{-1/2}qp^{-1/2})^tp^{1/2}.$$ The exponential map of ${\rm GL}_2^+(\mathcal{H})$ is given by $$Exp_p: T_p{\rm GL}_2^+(\mathcal{H}) \rightarrow {\rm GL}_2^+(\mathcal{H}), \ Exp_p(v)=p^{1/2}\exp(p^{-1/2}vp^{-1/2})p^{1/2}.$$
In \cite{Larotonda} G. Larotonda obtained general geometric results about ${\rm GL}_2^+(\mathcal{H})$ with the above metric: Riemannian conection, geodesic, sectional curvature, convexity of geodesic distance and completeness. 

\section{A polar Riemannian structure on ${\rm GL_2}(\mathcal{H})$}
The polar decomposition of $g \in {\rm GL_2}(\mathcal{H})$ induces a diffeomorphism 
\begin{align}
{\rm GL_2}(\mathcal{H}) &\stackrel{\varphi}\longrightarrow  {\rm U}_2(\mathcal{H}) \times {\rm GL}_2^+(\mathcal{H}) \label{difeopolar}\\ 
                g                    &\longmapsto  (u,\vert g \vert).\nonumber\
\end{align} 
This fact was noted in Prop.14 (iv) on page 98 of the book \cite{Harpe}. If we put the left invariant metric on ${\rm U}_2(\mathcal{H})$ and the positive metric (\ref{metpos}) on ${\rm GL}_2^+(\mathcal{H})$, then we can endow the product manifold ${\rm U}_2(\mathcal{H}) \times {\rm GL}_2^+(\mathcal{H})$ with the usual product metric, that is: if $v=(x,y) \in T_u {\rm U}_2(\mathcal{H})\times T_{\vert g\vert}{\rm GL}_2^+(\mathcal{H})$ we put 
\begin{eqnarray}\label{metpolar}
\mathcal{P}\big((u,\vert g\vert),v\big) & := &\bigg(\mathcal{I}(u,x)^2 + \mathfrak{p}(\vert g\vert,y)^2\bigg)^{1/2} \nonumber\\
                                                           & = &\bigg(\Vert x\Vert_2^2+ \Vert \vert g\vert^{-1/2}y\vert g\vert^{-1/2}\Vert_2^2\bigg)^{1/2}. 
\end{eqnarray}

This is the product metric in the Riemannian manifold ${\rm U}_2(\mathcal{H}) \times {\rm GL}_2^+(\mathcal{H})$. The map $\varphi$ is an immersion, from this we can define a new Riemannian metric in the group in the following way: if $v,w \in T_g{\rm GL}_2(\mathcal{H})$ we put $$\langle v,w\rangle_g:=\langle d\varphi_g(v) , d\varphi_g(w)\rangle_{(u,\vert g \vert)}.$$ It is clear that $\varphi$ is an isometric map with the above metric and if $\alpha$ is any curve in the group ${\rm GL}_2(\mathcal{H})$ we can measure its length as $L_{\mathcal{P}}(\varphi\circ\alpha)$.   

\begin{teo} \label{minpolar} Let $g\in {\rm GL}_2(\mathcal{H})$ with polar decomposition $u\vert g\vert$ and suppose that $u=e^x$ with $x \in \mathcal{B}_2(\mathcal{H})_{ah}$ and $\Vert x\Vert\leq \pi$, then the curve $\alpha(t)=e^{tx}\vert g\vert^t\subset {\rm GL}_2(\mathcal{H})$ has minimal length among all curves joining $1$ to $g$, if we endow ${\rm GL}_2(\mathcal{H})$ with the polar Riemannian metric (\ref{metpolar}).
\end{teo}

\begin{proof} By the polar decomposition, $\varphi\circ \alpha(t)=(e^{tx},\vert g\vert^t)$ and its length is $$L_{\mathcal{P}}(\varphi\circ \alpha)=\int_0^{1}        \mathcal{P}\big((e^{tx},\vert g\vert^t),(xe^{tx},\ln \vert g\vert \vert g\vert^{t})\big)dt=\big({\| x \|}^{2}_2 + {\|\ln \vert g\vert\|}^{2}_2\big)^{1/2}.$$ Let $\beta$ be another curve that joins the same endpoints and suppose that $\beta=\beta_1\beta_2$ is its polar decomposition where  $\beta_1\subset {\rm U}_2(\mathcal{H})$ and $\beta_2\subset {\rm GL}_2^+(\mathcal{H})$, then $$L_{\mathcal{P}}(\varphi\circ \beta)=\int_0^{1} \mathcal{P}\big((\beta_1,\beta_2), (\dot{\beta_1}, \dot{\beta_2})\big)dt = \int_0^{1} \big( \mathcal{I}(\beta_1,\dot{\beta_1})^{2} + \mathfrak{p}(\beta_2,\dot{\beta_2})^{2}\big)^{1/2} dt. $$ Using the Minkowski inequality (see inequality 201 of \cite{Hardy}) we have,
 \begin{align}\label{desunpos}  \int_0^{1} \bigl( \mathcal{I}(\beta_1,\dot{\beta_1})^{2} + \mathfrak{p}(\beta_2,\dot{\beta_2})^{2} \bigl)^{1/2} dt &\geq \biggl( \biggl\lbrace\int_0^{1} \mathcal{I}(\beta_1,\dot{\beta_1})\biggl\rbrace^{2} + \biggl\lbrace\int_0^{1} \mathfrak{p}(\beta_2,\dot{\beta_2})\biggl\rbrace^{2}\biggl)^{1/2}\nonumber\\
& = \biggl(L_{\mathcal{I}}(\beta_1)^{2}+L_{\mathfrak{p}}(\beta_2)^{2}\biggl)^{1/2}.
\end{align}
 It is known that the geodesic curve $e^{tx}$ has minimal length among all smooth curves in 
${\rm U}_2(\mathcal{H})$ joining the same endpoints (see \cite{Andruchow2}); using this fact and since the curve $\vert g\vert^{t}$ has minimal length with the positive metric $\mathfrak{p}$ (see \cite{Larotonda}) we have, $$L_{\mathcal{I}}(\beta_1)\geq L_{\mathcal{I}}(e^{tx})=\|x\|_2 \ \ \mbox{and} \ \ L_{\mathfrak{p}}(\beta_2)\geq L_{\mathfrak{p}}(e^{t\ln(\vert g \vert})=\|\ln \vert g \vert\|_2$$ then it is clear that $L_{\mathcal{P}}(\varphi\circ \beta)\geq L_{\mathcal{P}}(\varphi\circ \alpha)$.
\end{proof}

\begin{rem} Let $p,q \in {\rm GL}_2(\mathcal{H})$, suppose that $u_p\vert p\vert$ and $u_q\vert q\vert$ are their polar decompositions. From the surjectivity of the exponential map we can choose $z \in \mathcal{B}_2(\mathcal{H})_{ah}$ such that $u_q=u_pe^{z}$ with $\|z\|\leq \pi$. Then the curve
$$\alpha_{p,q}(t)=u_pe^{tz}\vert p\vert^{1/2}({\vert p\vert}^{-1/2}\vert q\vert {\vert p\vert}^{-1/2})^t\vert p\vert^{1/2}\subset  {\rm GL}_2(\mathcal{H})$$ has minimal length among all curves joining $p$ to $q$.  
\end{rem}
The above fact shows that the curve $\alpha_{p,q}$ is a geodesic of the Levi-Civita connection of the polar metric. Its length is $$\bigg({\|z\|}^{2}_2 + {\|\ln \vert p\vert^{-1/2} \vert q\vert\vert p\vert^{-1/2}\|}^{2}_2\bigg)^{1/2}. $$ From this, the geodesic distance is
$$d_{\mathcal{P}}(p,q)=\big(d_{\mathcal{I}}(u_p,u_q)^2+d_{\mathfrak{p}}(\vert p\vert,\vert q \vert)^2\big)^{1/2}.$$

\begin{prop} The metric space $( {{\rm GL}_2(\mathcal{H})},d_\mathcal{P})$ is complete.
\end{prop}

\begin{proof} Let $(x_n)\subset {\rm GL}_2(\mathcal{H})$ be a Cauchy sequence with $d_{\mathcal{P}}$, if $x_n=u_{x_n}\vert x_n\vert$ is its polar decomposition, we have that $$d_{\mathcal{I}}(u_{x_n},u_{x_m})\leq d_{\mathcal{P}}(x_n,x_m)=\big(d_{\mathcal{I}}(u_{x_n},u_{x_m})^2+d_{\mathfrak{p}}(\vert x_n\vert,\vert x_m \vert)^2\big)^{1/2}$$ then the unitary part is a Cauchy sequence in $({\rm U}_2(\mathcal{H}),d_{\mathcal{I}})$  and by \cite{Andruchow2} it is  $d_{\mathcal{I}}$ convergent to an element $u \in {\rm U}_2(\mathcal{H})$. Analogously the positive part is a Cauchy sequence in $({\rm GL}_2^+(\mathcal{H}),d_{\mathfrak{p}})$ then it is convergent to an element $g\in {\rm GL}_2^+(\mathcal{H})$ (see \cite{Larotonda}). If we put $x:=ug \in {\rm GL}_2(\mathcal{H})$  then, $$d_{\mathcal{P}}(x_n,x)=\big(d_{\mathcal{I}}(u_{x_n},u)^2+d_{\mathfrak{p}}(\vert x_n\vert,g)^2\big)^{1/2} \rightarrow 0.$$         
\end{proof}

In the next proposition we will compare the geodesic distance measured with the polar metric versus the left invariant metric. 

\begin{prop} \label{despol} Given $p,q \in {\rm GL}_2(\mathcal{H})$, if we denote $v:=\vert p\vert^{-1/2} \vert q\vert\vert p\vert^{-1/2}$ we can estimate the geodesic distance $d_{\mathcal{I}}$ by the geodesic distance $d_{\mathcal{P}}$ as, $$d_{\mathcal{I}}(p,q)\leq c(p,q)d_{\mathcal{P}}(p,q)$$ where $$c(p,q)^2=2\max\big\lbrace \ e^{4\|\ln(v)\|}  \big(\| p \|  \| p^{-1} \|\big)^2 ,   \| p \|  \| p^{-1}\| \big\rbrace.$$  
\end{prop}
\begin{proof} The proof is similar to Proposition 6.4 in \cite{Lopez Galvan}. If we differentiate ${\alpha}_{p,q}$ we have, $$\dot{\alpha}_{p,q}=u_pze^{tz}\vert p\vert^{1/2}e^{t\ln(v)}\vert p\vert^{1/2}+u_pe^{tz}\vert p\vert^{1/2}\ln(v)e^{t\ln(v)}\vert p\vert^{1/2}$$ and the inverse of the curve ${\alpha}_{p,q}$ is $$\alpha^{-1}_{p,q}=\vert p\vert^{-1/2}e^{-t\ln(v)}\vert p\vert^{-1/2}e^{-tz}u_p^{-1}.$$ After some simplifications we can write $$\alpha^{-1}_{p,q}\dot{\alpha}_{p,q}=\vert p\vert^{-1/2}e^{-t\ln(v)}\vert p\vert^{-1/2}z\vert p\vert^{1/2}e^{t\ln(v)}\vert p\vert^{1/2}+\vert p\vert^{-1/2}\ln(v)\vert p\vert^{1/2}.$$ Let  $x:=\vert p\vert^{1/2}e^{t\ln(v)}\vert p\vert^{1/2},$ taking the norm and using the parallelogram rule we have, 
\begin{align}
\|\alpha^{-1}_{p,q}\dot{\alpha}_{p,q}\|^{2}_2 &= \|x^{-1}zx+\vert p\vert^{-1/2}\ln(v)\vert p\vert^{1/2}\|^{2}_2  \nonumber\\
& \leq 2\big (\|x^{-1}zx\|^{2}_2+\|\vert p\vert^{-1/2}\ln(v)\vert p\vert^{1/2}\|^{2}_2\big)  \nonumber\\
& \leq 2\big (\|x^{-1}\|^2 \ \|x\|^2 \ \|z\|^{2}_2+ \| \vert p \vert^{-1/2} \|^2 \ \| \ln(v) \|^{2}_2 \ \| \vert p\vert^{1/2}\|^2\big ) \label{des2} .
\end{align}
We can estimate  $\|x\|^2$ and $\|x^{-1}\|^2$ by $$\|x\|^2\leq \| \vert p\vert^{1/2}\|^4 \ e^{2\|\ln(v)\|}=\|p\|^2e^{2\|\ln(v)\|} $$
 and $$\|x^{-1}\|^2\leq \| \vert p\vert^{-1/2}\|^4 \ e^{2\|\ln(v)\|}=\|p^{-1}\|^2e^{2\|\ln(v)\|}.$$
If we define  $$c(p,q)^2=2\max \big\lbrace \ e^{4\|\ln(v)\|}  \big(\| p \|  \| p^{-1} \|\big)^2 ,   \| p \|  \| p^{-1}\| \big\rbrace. $$
from (\ref{des2}) and taking square roots we have, $$\|\alpha^{-1}_{p,q}\dot{\alpha}_{p,q}\|_2\leq c(p,q)\big(\|z\|^{2}_2+\| \ln(v) \|^{2}_2\big)^{1/2}=
c(p,q)d_{\mathcal{P}}(p,q),$$ then $$d_{\mathcal{I}}(p,q)\leq L_{\mathcal{I}}(\alpha_{p,q})\leq c(p,q)d_{\mathcal{P}}(p,q).$$

\end{proof}

\section{Riemannian geometry of a self-adjoint subgroup}

\subsection{Riemannian geometry with the left invariant metric}
It is known that the covariant derivative given by the left invariant Riemannian metric on the Banach-Lie group $ {\rm GL_2}(\mathcal{H})$ can be expressed as
\begin{equation}\label{conecinv}
\alpha^{-1}D_t\eta=\dot{\mu}+1/2\lbrace [\beta,\mu]+[\beta,\mu^{*}]+[\mu,\beta^{*}]\rbrace
\end{equation} 
 where $\alpha:(-\epsilon,\epsilon) \rightarrow  {\rm GL_2}(\mathcal{H})$ is any smooth curve, $\eta$ is a tangent field along $\alpha$ and $\beta=\alpha^{-1}\dot{\alpha}, \ \mu=\alpha^{-1}\eta \subset\mathcal{B}_2(\mathcal{H})$ are the fields at the identity (see \cite{Lopez Galvan} for a proof). 
\begin{prop}\label{geocompl} Let $G\subset {\rm GL}_2(\mathcal{H})$ with the left invariant metric (\ref{metinvizq}), then $G$ is totally geodesic. That is, if $\alpha \subset G$ is a curve and $\eta$ is a field along $\alpha$ then $$D_t\eta \in (TG)_\alpha=\alpha.\mathfrak{g}.$$
\end{prop}
\begin{proof} Let $\beta=\alpha^{-1}\dot{\alpha}$ and $\mu=\alpha^{-1}\eta$ be the fields at $\mathfrak{g}$, we will show that $\alpha^{-1}D_t\eta \subset \mathfrak{g}$. Indeed, since $\mu \subset \mathfrak{g}$ then it is clear that $\dot{\mu}$ belongs to $\mathfrak{g}$, because it is a limit of operators that belong to the closed algebra $\mathfrak{g}$. Since $G$ is self-adjoint, we have $\mathfrak{g}^*=\mathfrak{g}$, then $\mu^{*}$ and $\beta^{*}$ belong to $\mathfrak{g}$ and the brackets $[\beta,\mu],  [\beta,\mu^{*}],  [\mu,\beta^{*}]$ are all in $\mathfrak{g}$ because it is a closed Banach-Lie algebra. Then using equation (\ref{conecinv}) we have $\alpha^{-1}D_t\eta \subset \mathfrak{g}$.
\end{proof}

This shows that the Riemannian connection given by the left invariant metric in the group $G$ matches the one of ${\rm GL_2}(\mathcal{H})$. In particular, the geodesics of $G$ are the same than those of ${\rm GL_2}(\mathcal{H})$; if $g_0 \in G$ and $g_0v_0 \in g_0.\mathfrak{g}$ are the initial position and the initial velocity then  $$\alpha(t)=g_0e^{tv_0^{*}}e^{t(v_0-v_0^{*})}\subset G$$ satisfies $D_t\dot{\alpha}=0$ (see \cite{Andruchow1}). In this context the Riemannian exponential for $g \in G$ is $$Exp_g(v)=ge^{v^{*}}e^{v-v^{*}}$$ with $v \in \mathfrak{g}$. 

\subsection{Riemannian geometry with the polar metric}
The next theorem summarizes the most important properties of self-adjoint Banach-Lie groups. It was proved by G. Larotonda in \cite{Larotonda}. 
\begin{teo}\label{propself} Let $G=\langle \exp(\mathfrak{g})\rangle$ be a connected, self-adjoint Banach-Lie group with Banach-Lie algebra $\mathfrak{g}\subset \mathcal{B}_2(\mathcal{H})$. Let $P$ be the analytic map $g \mapsto g^*g$, $P :G \rightarrow G$. Let $\mathfrak{k}=\ker (d_1P)$, $\mathfrak{m}=\mbox{Ran}(d_1P)$. Let $M_G=\exp(\mathfrak{m})$ and $K=G\cap {\rm U}_2(\mathcal{H})=P^{-1}(1)$. Then \\

1. The set $\mathfrak{m}$ is a closed Lie triple system. We have  $[\mathfrak{m},\mathfrak{m}]\subset \mathfrak{l},  [\mathfrak{k},\mathfrak{m}]\subset \mathfrak{m},  [\mathfrak{k},\mathfrak{k}] \subset \mathfrak{k}$ and $\mathfrak{g}=\mathfrak{k}\oplus \mathfrak{m}$. In particular $\mathfrak{k}$ is a Banach-Lie subalgebra of $\mathfrak{g}$.\\

2. $P(G)=M_G$ and $M_G$ is a geodesically convex submanifold of ${\rm GL}_2^+(\mathcal{H})$.\\

3. For any $g=u_g\vert g \vert$ (polar decomposition), we have $\vert g\vert \in M_G$ and $u_g \in K$.\\

4. Let $g\in G$, $p \in M_G, I_g(p) = gpg^*$. Then $I_g \in I(M_G)$ (the group of isometries of $M_G$). If $g=p^{\frac{1}{2}}(p^{-\frac{1}{2}}qp^{-\frac{1}{2}})^{\frac{1}{2}}p^{\frac{1}{2}}$, then $I_g(p)=q$, namely $G$ acts isometrically and transitively on $M_G$.\\

5. Let $u \in K$ and $x \in \mathfrak{m}$ (resp. $\mathfrak{m}^{\perp}$). Then $I_u(x) = uxu^* \in \mathfrak{m}$ (resp. $\mathfrak{m}^{\perp}$). If $p, q \in M_G$ then $I_p$ maps $T_qM_G$ (resp. $T_qM_G^{\perp}$) isometrically onto $T_{I_p(q)}M_G$ (resp. $T_{I_p(q)}M_G^{\perp}$).\\

6. The group $K$ is a Banach-Lie subgroup of $G$ with Lie algebra $\mathfrak{k}$.\\

7. $G \simeq K \times  M_G$ as Hilbert manifolds. In particular $K$ is connected and $G/K \simeq M_G$.\\   
\end{teo}

Since $M_G = \exp(\mathfrak{m})$ is closed and a geodesically convex submanifold of ${\rm GL}_2^+(\mathcal{H})$,  then for any $p=e^x \in M_G$,  $$T_pM_G=Exp^{-1}_p(M_G)=\lbrace p^{1/2}\ln(p^{-1/2}qp^{-1/2})p^{1/2}: q \in M_G\rbrace.$$ For this reason it is clear, as in the case of the full space ${\rm GL}_2^+(\mathcal{H})$, that given any $p,q \in M_G$ the curve $$\gamma_{pq}(t)=p^{1/2}(p^{-1/2}qp^{-1/2})^tp^{1/2}\subset M_G$$ has minimal length among all curves in $M_G$ that join $p$ to $q$. Its length is measured with the metric (\ref{metpos}) and it is $\Vert \ln(p^{-1/2}qp^{-1/2})\Vert_2$.
Moreover the metric space $(M_G,d_{\mathfrak{p}})$ is complete.

The diffeomorphism given in point seven of the Theorem \ref{propself} is the restriction of the map (\ref{difeopolar}) on $G$. So, we can endow $G$ with the polar Riemannian metric (\ref{metpolar}) using the product manifold  $K \times  M_G$. 

If we consider a curve $\alpha \subset K$ and $\eta$ is a tangent field along $\alpha$ then the covariant derivative (\ref{conecinv}) given by the left invariant metric is reduced to $$\alpha^{-1}D_t\eta=\dot{\mu}+1/2[\beta,\mu] \in \mathfrak{k}.$$ Therefore, the geodesics of the Banach-Lie subgroup $K$ with the left invariant metric are one-parameter groups. 

Using the above facts, it is not difficult to see that given any initial velocity $v\in \mathfrak{g}$, if $v=x+y$ is the decomposition into $\mathfrak{k}\oplus \mathfrak{m}$, then the geodesics of the polar metric starting at the identity are $$\alpha(t)=e^{tx}e^{ty}.$$

\begin{prop} The geodesics of the left invariant metric coincide with the geodesics of the polar metric if the initial velocity $v \in \mathfrak{g}$ is normal.
\end{prop}
\begin{proof}
Let $v=x+y \in \mathfrak{k}\oplus \mathfrak{m}$ be the decomposition into its hermitic and anti-hermitic part, since $v$ is normal a straightforward computation shows that $x$ commutes with $y$, thus we have $$e^{tv^{*}}e^{t(v-v^{*})}=e^{tv}=e^{tx}e^{ty}.$$This equation shows that the geodesics are one-parameter groups.    
\end{proof}

\section{Completeness of the geodesic distance}
\subsection{Completeness in finite dimension with p-norms}
Let ${\rm GL}_n(\mathbb{C})$ be the general linear group in finite dimension. Let $\mathcal{M}_n(\mathbb{C})$ be the space of $n\times n$ complex matrices. Since ${\rm GL}_n(\mathbb{C})$ is open in the space $\mathcal{M}_n(\mathbb{C})$, we can identify the tangent space of ${\rm GL}_n(\mathbb{C})$ at any point with $\mathcal{M}_n(\mathbb{C})$. In this algebra we consider the classical {\it p-norms}, if $x\in \mathcal{M}_n(\mathbb{C})$ and $\tau$ denote the real part of the trace we put, $$\Vert x\Vert_p^p=\tau\big((x^*x)^{p/2}\big)  \  \mbox{for any} \ p\geq 1.$$
Since the dimension of ${\rm GL}_n(\mathbb{C})$ is finite, it is known that any closed subgroup $G$ has a structure of Banach-Lie subgroup of ${\rm GL}_n(\mathbb{C})$. In this context we denote the left invariant metric for any self-adjoint closed subgroup as $\mathcal{I}_p(g,v)=\Vert g^{-1}v\Vert_p$ for $g\in G$ and $v\in T_gG$.
\begin{teo} The space $(G,d_{\mathcal{I}_p})$ is a complete metric space.
\end{teo}
\begin{proof} If $p=2$, by Hopf-Rinow's theorem, the space $(G,d_{\mathcal{I}_2})$ is complete since the manifold $G$ is geodesically complete with the 2-norm (Proposition \ref{geocompl}). Now, we claim that $d_{\mathcal{I}_p}$ is equivalent to $d_{\mathcal{I}_2}$ for any $p\geq 2$. Indeed, at each tangent space of $G$, the $p$-norm is equivalent with the 2-norm with constants which depend only on the dimension of $\mathcal{M}_n(\mathbb{C})$. Examining the length functionals, it follows that the metrics are equivalent, with the same constants.
\end{proof}

\subsection{Completeness in the infinite dimensional case}

Since the map $g\mapsto g\vert g \vert ^{-1}=u_g$ is continuous, it is clear that if $p,q \in G$ are close its unitary parts ($u_p, u_q$) are close too. So, if $g \in G$ is close to the identity 1, then its unitary part $u_g$ is close to 1 too and since the polar decomposition is in the group (Theorem \ref{propself}) we can assume that $u_g$ lies in $K\cap U$ where $U$ is a neighbourhood of the identity in $G$. Since $K$ is a Banach-Lie subgroup of $G$, if we reduce the neighbourhood $U$, we can assume that $u_g=\exp(z)$ where $z$ belongs to a neighbourhood of $0$ in the Lie algebra $\mathfrak{k}\subset \mathfrak{g}$.
Now, if $p,q \in G$ are close and $u_p\vert p\vert, u_q\vert q\vert$ are their polar decompositions, then the unitary element $u_p^{-1}u_q \in K$ is close to 1 and we can choose an element $z\in \mathfrak{k}\subset \mathfrak{g}$ such that $u_p^{-1}u_q=e^z$. So, we can build the following smooth curve in $G$;
\begin{equation}\label{alphapq} 
\alpha_{p,q}(t)=u_pe^{tz}\vert p\vert^{1/2}({\vert p\vert}^{-1/2}\vert q\vert {\vert p\vert}^{-1/2})^t\vert p\vert^{1/2}\subset G
\end{equation}
that joins $p$ to $q$. This curve will be the key of the completeness with both metrics. 

\begin{teo} The metric space $(G,d_{\mathcal{P}})$ is complete.
\end{teo}
\begin{proof} Let $(x_n)$ be a $d_{\mathcal{P}}$-Cauchy sequence, let $x_n=u_{x_n}\vert x_n\vert$ be its polar decomposition. First we will prove that $u_{x_n}\subset K$ and $\vert x_n\vert\subset M_G$ are $d_{\mathcal{I}}$ and $d_{\mathfrak{p}}$-Cauchy sequences respectively. Indeed, given $\varepsilon=1/n$ there exist curves $\beta_n \subset G$ such that $\beta_n(0)=x_n, \ \beta_n(1)=x_m$ and $d_{\mathcal{P}}(x_n,x_m)+1/n > L_{\mathcal{P}}(\beta_n)$. If $\beta_{1n}\subset K$ and $\beta_{2n}\subset M_G$ denote the unitary and positive part of $\beta_n$, then since $x_n=\beta_n(0)=\beta_{1n}(0)\beta_{2n}(0)=u_{x_n}\vert x_n\vert$, it is clear that $\beta_{1n}$ joins $u_{x_n}$ to $u_{x_m}$ and $\beta_{2n}$ joins $\vert x_n\vert$ to $\vert x_m\vert$. Using the inequality (\ref{desunpos}) we have, $$d_{\mathcal{P}}(x_n,x_m)+1/n > L_{\mathcal{P}}(\beta_n)\geq L_{\mathcal{I}}(\beta_{1n})\geq d_{\mathcal{I}}(u_{x_n},u_{x_m})$$ for all $n,m$, then it is clear that $d_{\mathcal{I}}(u_{x_n},u_{x_m})\rightarrow 0$ when $n,m\rightarrow \infty$. An analogous computation using $\beta_{2n}$ shows that $d_{\mathfrak{p}}(\vert x_n\vert,\vert x_m\vert)\rightarrow 0$ when $n,m\rightarrow \infty$. Since $(u_{x_n})$ is an unitary sequence it is known that $\Vert u_{x_n} - u_{x_m}\Vert_2\leq d_\mathcal{I}(u_{x_n}, u_{x_m})$ and therefore $(u_{x_n})\subset K$ is a 2-norm Cauchy sequence. Since $K$ is closed we can take $u\in K$ such that $ u_{x_n} \stackrel{\Vert .\Vert_2}\longrightarrow u$. Then if $n$ is large, we can suppose that $u^{-1}u_{x_n}$ is close to 1, therefore there exists a sequence $(z_n)\subseteq \mathfrak{k}$ such that $u^{-1}u_{x_n}=e^{z_n}$ and $z_n \stackrel{\Vert .\Vert_2}\longrightarrow 0$. On the other hand, there exists $g\in M_G$ such that $d_{\mathfrak{p}}(\vert x_n \vert,g)=\Vert\ln(g^{-1/2}\vert x_n\vert g^{-1/2})\Vert_2 \rightarrow 0$. It is clear that $ug\in G$, then if $n$ is large we can consider the curve $\alpha_{ug,x_n}(t)=ue^{tz_n}g^{1/2}(g^{-1/2}\vert x_n\vert g^{-1/2})^tg^{1/2}$ (\ref{alphapq}) that joins $x_n$ and $ug$, therefore we have 
$$d_{\mathcal{P}}(x_n,ug)\leq L_{\mathcal{P}}(\alpha_{ug,x_n})=\big ( \Vert z_n\Vert_2^2+\Vert \ln(g^{-1/2}\vert x_n\vert g^{-1/2})\Vert_2^2 \big)^{1/2} \rightarrow 0.$$
\end{proof}

The following proposition is a generalization of Proposition \ref{despol} on $G$.  
\begin{prop}\label{des} Suppose $p,q \in G$ are close and let $v:=\vert p\vert^{-1/2} \vert q\vert\vert p\vert^{-1/2}$ then we can estimate the geodesic distance $d_{\mathcal{I}}$ by $$d_{\mathcal{I}}(p,q)\leq c(p,q)\big(\|z\|^{2}_2+\| \ln(v) \|^{2}_2\big)^{1/2}$$ where $$c(p,q)^2=2\max\big\lbrace \ e^{4\|\ln(v)\|}  \big(\| p \|  \| p^{-1} \|\big)^2 ,   \| p \|  \| p^{-1}\| \big\rbrace.$$ 
\end{prop}
\begin{proof} The proof is similar to Proposition \ref{despol}. Since $p,q \in G$ are close, then we can build the smooth curve ${\alpha}_{p,q}\subset G$ (\ref{alphapq}) that joins $p$ to $q$; so we can repeat the argument that we gave in Proposition \ref{despol}. Then in this case we have, $$\|\alpha^{-1}_{p,q}\dot{\alpha}_{p,q}\|_2\leq c(p,q)\big(\|z\|^{2}_2+\| \ln(v) \|^{2}_2\big)^{1/2}.$$
and then $d_{\mathcal{I}}(p,q)\leq L_{\mathcal{I}}(\alpha_{p,q})\leq  c(p,q)\big(\|z\|^{2}_2+\| \ln(v) \|^{2}_2\big)^{1/2}$.
\end{proof}
\begin{lem}  If $(x_n)\subset G$ is a Cauchy sequence in $(G,d_{\mathcal{I}})$ then it is a Cauchy sequence in  $(G,\|. \|_2)$.
\end{lem}
\begin{proof} Since the geodesics of the Riemannian connection are of the form $\alpha(t)=Exp_g(tv)=ge^{tv^{*}}e^{t(v-v^{*})}$, the proof of this lemma can be adapted easily from Lemma 7.1 in \cite{Lopez Galvan}. 
\end{proof}
Now we are in position to obtain our final result.
\begin{teo} The metric space $(G,d_{\mathcal{I}})$ is complete.
\end{teo}
\begin{proof}  Let $(x_n)\subset G$ be a $d_{\mathcal{I}}$-Cauchy sequence, by the above lemma it is $\Vert .\Vert_2$-Cauchy; then since $G$ is closed there exists $x \in G$ such that $x_n \stackrel{\|.\|_2}\longrightarrow x$. Now we will show that $x_n \stackrel{d_{\mathcal{I}}} \longrightarrow x$; indeed from the continuity of the module we have that $\vert x_n \vert$ converges to $\vert x\vert$ in $\|.\|_2$ and its unitary part $u_{x_n}=x_n\vert x_n\vert^{-1}$ converges to $u_x=x\vert x\vert^{-1}$. The sequence $\vert x\vert^{-1/2}\vert x_n\vert\vert x\vert^{-1/2}$ converges to $1$  and then $\|\ln(\vert x\vert^{-1/2}\vert x_n\vert\vert x\vert^{-1/2})\|_2\rightarrow 0.$ Since $x_n$ converges to $x$, we can assume that $x_n$ is close to $x$ if $n\geq n_0$, therefore we can use Proposition \ref{des} to estimate the geodesic distance. Then we have $$d_{\mathcal{I}}(x,x_n)\leq c(x,x_n) \big(\|z_n\|^{2}_2+\| \ln(\vert x\vert^{-1/2}\vert x_n\vert\vert x\vert^{-1/2}) \|^{2}_2\big)^{1/2}$$ where $z_n \in \mathfrak{k}\subset \mathfrak{g}$ is such that $u_x^{-1}u_{x_n}=e^{z_n}$ (since  $u_x^{-1}u_{x_n}$ is close to 1 and $K$ is a Banach-Lie subgroup of $G$). We also have $\Vert z_n\Vert_2 \rightarrow 0$. Now we will see that $c(x,x_n)$ is uniformly bounded. Indeed, since $\|\ln(\vert x\vert^{-1/2}\vert x_n\vert\vert x\vert^{-1/2})\|\leq\|\ln(\vert x\vert^{-1/2}\vert x_n\vert\vert x\vert^{-1/2})\|_2 \rightarrow 0$, then for $n$ large we can assume that $\|\ln(v_n)\|\leq 1$ where $v_n=\vert x\vert^{-1/2}\vert x_n\vert \vert x \vert^{-1/2}$ as we denoted in Proposition \ref{des}. Finally we have $$c(x,x_n)^2=2\max \lbrace 
   \ e^{4\|\ln(v_n)\|}   \big( \|x\| \|x^{-1}\|\big)^2, \ \|x\| \|x^{-1}\| \rbrace$$ $$\leq 2\max \lbrace  \ e^{4} \big( \|x\| \|x^{-1}\|\big)^2 , \ \|x\| \|x^{-1}\| \rbrace$$ is clearly uniformly bounded and then it is clear that $d_{\mathcal{I}}(x,x_n)\rightarrow 0$.
\end{proof}

\paragraph{Acknowledgements}
I would like to thank Prof. G. Larotonda for his suggestions and support.

\bigskip
{\footnotesize Alberto Manuel L\'opez Galv\'an.\\
Instituto Argentino de Matem\'atica "Alberto P. Calder\'on".\\
Saavedra 15 (C.P. 1083). Buenos Aires, Argentina.\\
e-mail: mlopezgalvan@hotmail.com}

\end{document}